\documentclass[12pt]{article}
\usepackage{latexsym,color,graphicx,amsmath,amssymb,amsthm}
\usepackage[a4paper, margin=2.8cm]{geometry}
\usepackage{hyperref}
\def\R{{\mathbb R}}
\def\C{{\mathbb C}}

\def\minus{\backslash}
\DeclareMathOperator{\cl}{Cl}

\def\marrow{\marginpar[\hfill$\longrightarrow$]{$\longleftarrow$}}

\def\orit#1{\textsc{(Orit says: }\marrow\textsf{#1})}

\newtheorem{theorem}{Theorem}[section]
\newtheorem{lemma}[theorem]{Lemma}
\newtheorem{proposition}[theorem]{Proposition}
\newtheorem{corollary}[theorem]{Corollary}

\theoremstyle{definition}
\newtheorem{definition}[theorem]{Definition}

\begin{document}

\title{The Elekes-Szab\'o Theorem in four dimensions\let\thefootnote\relax\footnotetext{
Work on this paper by Orit E. Raz and Micha Sharir was supported by Grant 892/13 from the Israel Science Foundation, and
by the Israeli Centers of Research Excellence (I-CORE) program (Center No.~4/11).
Work by Orit E. Raz was also supported by a Shulamit Aloni Fellowship from the Israeli Ministry of Science. 
Work by Micha Sharir was also supported by Grant 2012/229 from the U.S.--Israel Binational Science Foundation, 
and by the Hermann Minkowski-MINERVA Center for Geometry at Tel Aviv University.
Work on this paper by Frank de Zeeuw was partially supported by Swiss National Science Foundation Grants 200020-165977 and 200021-162884.}}

\author{
Orit E. Raz\thanks{%
School of Computer Science, Tel Aviv University,
Tel Aviv
, Israel.
{\sl oritraz@post.tau.ac.il} }
\and
Micha Sharir\thanks{%
School of Computer Science, Tel Aviv University,
Tel Aviv
, Israel.
{\sl michas@post.tau.ac.il} }
\and
Frank de Zeeuw\thanks{%
Mathematics Department, EPFL, Lausanne, Switzerland.
{\sl fdezeeuw@gmail.com}} }

\maketitle

\vspace{-20pt}
\begin{abstract}
Let $F\in\C[x,y,s,t]$ be an irreducible constant-degree polynomial,
and let $A,B,C,D\subset\C$ be finite sets of size $n$.
We show that $F$ vanishes on at most $O(n^{8/3})$ points of the 
Cartesian product $A\times B\times C\times D$,
unless $F$ has a special group-related form. 
A similar statement holds for $A,B,C,D$ of unequal sizes,
with a suitably modified bound on the number
of zeros. 
This is a four-dimensional extension of our recent improved 
analysis of the original Elekes-Szab\'o theorem in three dimensions.
We give three applications: an expansion bound for three-variable 
real polynomials that do not have a special form, 
a bound on the number of coplanar quadruples on a 
space curve that is neither planar nor quartic, 
and a bound on the number of four-point circles on 
a plane curve that has degree at least five.
\end{abstract}

\section{Introduction}

Elekes and R\'onyai \cite{ER00} and Elekes and Szab\'o \cite{ES12} 
pioneered the study of algebraic structures behind problems from combinatorial geometry. 
The main result of \cite{ES12} was quantitatively improved by the authors in \cite{RSZ15} to the following statement;
we state it here in a somewhat rough form, and refer to \cite{RSZ15} for a full and precise statement.
Given an irreducible polynomial $F\in \C[x,y,z]$ and finite sets $A,B,C\subset \C$ of size $n$, 
we have the bound (writing $Z(F)$ for the zero set of $F$)
\begin{align}\label{eq:ES3D}
|Z(F)\cap (A\times B\times C)| = O(n^{11/6}),
\end{align}
with the constant of proportionality depending on the degree of $F$, 
unless $F$ has the special property that, 
in a certain local sense, the equation $F(x,y,z)=0$ is equivalent to $\varphi_1(x)+\varphi_2(y)+\varphi_3(z)=0$, 
for some locally defined invertible analytic functions $\varphi_1,\varphi_2,\varphi_3$.
When $F$ does not have this special property,
the bound \eqref{eq:ES3D} improves  on the simple bound $O(n^2)$ 
that follows from the Schwartz-Zippel lemma \cite{Sch80,Zi89} 
(see also \cite[Lemma A.4]{RSZ15}).
A bound similar to \eqref{eq:ES3D} holds in the ``unbalanced case'', 
when the sets $A,B,C$ are allowed to have different sizes.
We refer to \cite{ES12, RSS14, RSZ15, Z16} for further background, 
including many examples of problems from combinatorial 
geometry that fit into to this algebraic framework.

In the current paper we prove a natural four-dimensional variant of the result of \cite{ES12,RSZ15}.
The only previous work that considered such a variant is Schwartz et al.~\cite{SSZ13},
where a weaker bound was proved in the special case $F = f(x,y,s)-t$, for a real polynomial $f$ that is not of the form $g(h(x)+k(y)+l(z))$ or $g(h(x)\cdot k(y)\cdot l(z))$ for real polynomials $g,h,k,l$.
Here we prove a quantitatively better bound for arbitrary complex polynomials $F(x,y,s,t)$, with an exceptional form analogous to that in \cite{ES12,RSZ15}.
In Corollary \ref{cor:ER4D}, we deduce an improved version of the result of \cite{SSZ13}.

As in \cite{RSZ15}, we state our main theorem in an ``unbalanced'' form, where the finite sets are permitted to have different sizes.

\begin{theorem}\label{thm:main}
Let $F\in \C[x,y,s,t]$ be irreducible of degree $\delta$, with none of $F_x,F_y,F_s,F_t$ identically zero.
Then one of the following two statements holds.\\
(i) For any finite sets $A,B,C,D\subset \C$ we have 
\begin{align*}
|Z(F)\cap (A\times B\times C\times D)| = O\Big(|A|^{2/3}|B|^{2/3}|C|^{2/3}|D|^{2/3}+|A||B|+ |A||C|\\+|A||D|+|B||C|+|B||D|+|C||D|\Big),
\end{align*}
where the constant of proportionality depends (polynomially) on $\delta$.\\
(ii) There exists a two-dimensional subvariety
$Z_0\subset Z(F)$,
such that for all $v\in Z(F)\minus Z_0$,
there exist open sets $D_i\subset \C$ and one-to-one analytic functions $\varphi_i: D_i\to \C$, with analytic inverses, for $i=1,2,3,4$,
such that $v\in D_1\times D_2\times D_3\times D_4$ and for all $(x,y,s,t)\in D_1\times D_2\times D_3\times D_4$ we have
\[(x,y,s,t)\in Z(F)~~~\text{if and only if}~~~\varphi_1(x)+\varphi_2(y)+\varphi_3(s) +\varphi_4(t)= 0.\]
\end{theorem}
In the ``balanced'' version, where $|A|=|B|=|C|=|D|=n$, the bound in $(i)$ becomes $O(n^{8/3})$.
This improves (for non-exceptional $F$) on the simple bound $O(n^3)$ that follows from the Schwartz-Zippel lemma.

The theorem of course also applies to real polynomials $F\in \R[x,y,s,t]$ and real sets $A,B,C,D\subset \R$.
However, the functions $\varphi_i$ in statement $(ii)$ would still be complex, 
and a priori there need not be real functions for which the same statement holds.
The same situation arose in \cite{RSZ15}, 
and there we proved a fully real analogue of the main theorem.
Using the same reasoning, we can obtain a real analogue of 
Theorem \ref{thm:main}, with real-analytic functions $\varphi_i$.
We omit the proof, 
which would be almost identical to that in \cite[Section 5]{RSZ15} 
(one only needs to change the number of variables).

%

\newpage
\paragraph{Discussion.} The proof of Theorem \ref{thm:main} proceeds along the same lines as the proof in \cite{RSZ15}:
 The bound in $(i)$ can be deduced from a planar incidence bound for points and algebraic curves, 
 unless the obtained points and curves are degenerate in a certain sense;
 this degeneracy implies that $F$ satisfies a certain differential equation,
 which then leads to the special property in $(ii)$.
 In fact, the proof of Theorem \ref{thm:main} is in some ways \emph{simpler} than that in \cite{RSZ15}.
 In particular, the curves can be defined in a more straightforward way (each curve is the zero set of $F(x,y,s_0,t_0)$ for fixed $s_0,t_0$)
 than in \cite{RSZ15} (where the curves were defined using quantifier elimination).
On the other hand, the argument leading from the differential equation for $F$ to the property $(ii)$ is more complicated in the current case than in \cite{RSZ15}.

One might view Theorem~\ref{thm:main} as 
providing an alternative context within which certain planar incidence problems can be studied. 
Specifically, in several algebraic applications of such incidence bounds
(as is also the case in our reduction),
the families of curves that arise are
defined by a single polynomial $F(x,y,s,t)$, 
with each curve in the family obtained from $F$ by fixing values of $s$ and $t$, 
using points $(s,t)$ from a given point set.
A subset of these applications has the further feature that the point set is 
a Cartesian product $A\times B$ of two finite sets $A,B$ in $\R$ or $\C$, 
and the curve set is parameterized by a Cartesian 
product $S\times T$ with $S,T$ finite sets in $\R$ or $\C$;
see Solymosi and De Zeeuw~\cite[Section 7]{SZ15} 
for several examples of such applications.
The incidence bound proved in \cite{SZ15}, of which 
Proposition \ref{prop:incbound} below is a corollary, 
is tailored to such applications,
but it comes with a combinatorial condition on the 
incidence structure of the points and curves 
(similar to the condition in the original point-curve incidence bound of 
Pach and Sharir \cite{PS98}).
Theorem \ref{thm:main} replaces this combinatorial 
condition with an algebraic condition on $F$.
An incidence bound with a  similar (but different) 
algebraic condition was obtained in \cite{NPVZ15}, 
for point sets and curve sets that do not have this 
Cartesian product structure. 

\paragraph{Applications.} 
One application of Theorem \ref{thm:main} is a higher-dimensional variant 
of the original result of Elekes and R\'onyai \cite{ER00} on expanding polynomials.
The result of \cite{ER00} (following the improvement in \cite{RSS14})
states that for $f\in \R[x,y]$ and finite sets $A,B\subset \R$ of size $n$,
we have $|f(A\times B)| = \Omega(n^{4/3})$ (with the 
constant of proportionality depending on the degree of $f$), 
unless $f$ has one of the special forms $g(h(x)+k(y))$ or 
$g(h(x)\cdot k(y))$, for $g,h,k\in \R[t]$ (in this
latter case, $|f(A\times B)|$ can be $O(n)$).
Schwartz, Solymosi, and De Zeeuw \cite{SSZ13} proved a three-variable version, 
which states that for $f\in \R[x,y,z]$ and finite sets $A,B,C\subset \R$ of size $n$, 
we have $|f(A\times B\times C)| = \omega(n)$, unless $f=g(h(x)+k(y)+l(z))$ or 
$g(h(x)\cdot k(y)\cdot l(z))$, for $g,h,k,l\in \R[t]$.
Using Theorem \ref{thm:main}, 
we can deduce a quantitative improvement of this bound.

\begin{corollary}\label{cor:ER4D}
Let $f\in \R[x,y,z]$ and let $A,B,C\subset \R$ be finite sets of size $n$.
Either we have
\[|f(A\times B\times C)| = \Omega(n^{3/2}), \]
with the constant of proportionality depending on the degree of $f$,
or $f$ has one of the forms 
\[f=g(h(x)+k(y)+l(z))~~~\text{or}~~~g(h(x)\cdot k(y)\cdot l(z)),\]
for some polynomials $g,h,k,l\in \R[t]$.
\end{corollary}
\begin{proof}
Set $F = t - f(x,y,s)$ and $D = f(A\times B\times C)$, 
and apply Theorem \ref{thm:main} 
(more precisely, we apply the real version, which holds by the remark following the theorem).
So $F$ satisfies one of the properties $(i)$ or (the real version of) $(ii)$.
Suppose first that $(i)$ holds. 
Then we have 
\[n^3 = |Z(F)\cap (A\times B\times C\times D)| =  O(n^2\cdot |D|^{2/3}+n|D|),\] 
which leads to $|D|= \Omega(n^{3/2})$.

Suppose $f$ satisfies the real version of statement $(ii)$, i.e., 
we have $D_i\subset \R$ and real-analytic maps $\varphi_i:D_i\to \R$ for $i=1,2,3,4$,
so that, locally, we have $F(x,y,s,t)=0$ if and only if 
\[\varphi_1(x)+\varphi_2(y)+\varphi_3(s) +\varphi_4(t)= 0.\]
Then, for any $n$, 
we can choose finite sets $A,B,C\subset \R$ of size $n$, 
so that there are only $O(n)$ elements $t\in \R$ for 
which there exists $(a,b,c)\in A\times B\times C$ such that
$\varphi_1(a)+\varphi_2(b)+\varphi_3(c) +\varphi_4(t)= 0$. 
More concretely, 
we choose $A,B,C$ so that each of the sets $\varphi_1(A)$,
$\varphi_2(B)$, $\varphi_3(C)$ forms an arithmetic progression with the same common difference. 
It follows that $|f(A\times B\times C)|=O(n)$.
By \cite[Theorem 1.4]{SSZ13},
this implies that $f$ is of one of the forms
$f=g(h(x)+k(y)+l(y))$ or $g(h(x)\cdot k(y)\cdot l(y))$, for some polynomials $g,h,k,l\in \R[t]$.
\end{proof}

For the specific polynomial $f(x,y,z)=x+yz$, the bound in Corollary \ref{cor:ER4D} is a well-known consequence of the Szemer\'edi-Trotter theorem (usually stated as $|A+BC|=\Omega(n^{3/2})$, for $A,B,C\subset \R$ of size $n$).
As far as we are aware, the earliest explicit mention of this bound is in Tao and Vu \cite[Exercise 8.3.3]{TV06}.
K\'arolyi \cite{K09} observed the same bound for $f(x,y,z) = xy+yz+zx$.
We do not know of any three-variable polynomial $f$ for which a better bound than $|f(A\times B\times C)|=\Omega(n^{3/2})$ is known, 
although Roche-Newton \cite{RN16} (improving on a similar result in \cite{MRNS15}) proved that in the case $A=B=C$, the polynomial $f(x,y,z) = x(y+z)$ satisfies $|f(A\times A\times A)|=\Omega(|A|^{3/2+5/242})$.

As a second application, 
we prove a three-dimensional analogue of a 
geometric application considered in \cite{RSZ15},
where we have shown that $n$ points on an algebraic curve in 
$\C^2$ span at most $O(n^{11/6})$ proper collinear triples (with the constant of 
proportionality depending on the degree of the curve), 
unless the curve contains a line or a cubic.
Here we prove the following analogous bound on 
coplanar quadruples on a curve in $\C^3$.
We say that a quadruple of points in $\C^3$ is 
\emph{proper} if the four points are distinct.

\begin{theorem}\label{thm:coplanarintro}
Let $C$ be an algebraic curve of degree $d$ in $\C^3$, and 
let $S\subset C$ be a finite set of size $n$. 
Then the number of proper coplanar quadruples from $S$ 
is
$O(n^{8/3})$, 
unless $C$ contains either a curve that is contained in a plane, or a curve of degree four.
\end{theorem}

Certain space curves of degree four have a group 
structure that allows us to construct subsets of $n$ points with 
$\Theta(n^3)$ coplanar quadruples (see Section \ref{sec:app}).
Space quartics also turn up in Ball \cite{B16} as curves with few {\it ordinary planes} (planes containing three points of a given point set in space).

Finally, we use Theorem \ref{thm:coplanarintro} to prove a 
bound on the number of four-point circles (circles incident to at least 
four points) spanned by a finite set on a plane algebraic curve, 
unless that curve is contained in a quartic curve (see Corollary 
\ref{cor:circles}).
This is related to the orchard problem for circles (see Lin et al. \cite{LMNSSZ16}).

The organization of our paper is as follows: In Section \ref{sec:proofoftheorem} we give the proof of Theorem \ref{thm:main}, 
except for the proof of the crucial Proposition \ref{prop:keyprop}, 
which is given in Section \ref{sec:proofofkey}.
In Section~\ref{sec:app} we prove Theorem \ref{thm:coplanarintro}
and its implication for four-point circles.

\section{Proof of Theorem \ref{thm:main}}\label{sec:proofoftheorem}

\subsection{Setup}

Let $F\in\C[x,y,s,t]$ be as in the statement of the theorem, and let $A,B,C,D\subset\C$ 
be four arbitrary finite sets.
The quantity we wish to bound is
$$
M:=|Z(F)\cap (A\times B\times C\times D)|.
$$
The strategy of the proof is to transform the problem of bounding $M$ into an incidence problem for points and curves in $\C^2$. 
The latter problem can then be tackled using the machinery that we have established in our recent work
\cite[Theorem 4.3]{RSZ15}, {\em provided} that 
the resulting curves have well-behaved intersections, in a sense that we will make precise below.
A major component of the proof is to show that 
if the points and curves that we are about to define 
do not have such well-behaved intersections, 
then $Z(F)$ must have the special form described in property $(ii)$ of the theorem.

\subsection{Curves}

\paragraph{Primal curves.}
For every point $(c,d)\in \C^2$, we define 
\begin{equation}\label{eq:gamma}
\gamma_{c,d}:= \{(x,y)\in \C^2\mid F(x,y,c,d) = 0\}.
\end{equation}
It is not always the case that $\gamma_{c,d}$ is a curve; it can turn out to be two-dimensional.
The following lemma quantifies the exceptional case, allowing us to exclude it in what follows.

\begin{lemma}\label{lem:quantelim}
Let $F\in \C[x,y,s,t]$ be an irreducible polynomial of degree $\delta$ such that none of $F_x, F_y,F_s,F_t$ is identically zero. 
Then there is a finite set $\mathcal{T}\subset \C^2$ with $|\mathcal{T}|\le \delta^2$ such that, for each $(c,d)\not\in \mathcal{T}$,
the set  $\gamma_{c,d}$, as defined in \eqref{eq:gamma}, is either an algebraic curve of degree at 
most $\delta$ or the empty set; for each $(c,d)\in \mathcal{T}$, the set $\gamma_{c,d}$ equals $\C^2$.
\end{lemma}
\noindent{\bf Proof.}
The set $\mathcal{T}$ we are after is
$$
\mathcal{T}:=\{(c,d)\in \C^2\mid F(x,y,c,d)\equiv 0~~\text{(as a polynomial in $x$ and $y$)}\}.
$$
Indeed,  for all $(c,d)\not\in \mathcal{T}$,
the set $\gamma_{c,d}$ is an algebraic curve of degree at most $\delta$ or the empty set,
while for $(c,d)\in \mathcal{T}$, the set $\gamma_{c,d}$ equals $\C^2$.
It remains to show that $\mathcal{T}$ is finite, and to establish 
the bound on the cardinality of $\mathcal{T}$.

Write 
\[F(x,y,s,t)=\sum_{i=0}^\delta\sum_{j=0}^{\delta-i} \alpha_{i,j}(s,t)x^iy^j,\]
with suitable bivariate polynomials $\alpha_{i,j}(s,t)$, each of degree at most $\delta$.
For $(c,d)\in \C^2$ we have $F(x,y,c,d)\equiv 0$ if and only if $\alpha_{i,j}(c,d) =0$ for all $i,j$.
That is,
$$
\mathcal{T}=\bigcap_{0\le i+j\le \delta}\{(s,t)\mid \alpha_{i,j}(s,t) =0\}.
$$
Observe that the existence of a common factor of the polynomials $\alpha_{i,j}$, for $0\le i+j\le \delta$, 
would contradict the irreducibility of $F$, or the assumption that none of its first-order derivatives is
identically zero.
By the same token, there must be at least two polynomials $\alpha_{i,j}$
that are not identically zero.
By an application of a variant of B\'ezout's inequality for many curves,
given as Lemma~3.10 in \cite{RSZ15},
we conclude that 
 $\mathcal{T}$ is finite and $|\mathcal{T}|\le \delta^2$.
$\hfill\square$

Let $\mathcal{T}\subset \C^2$ be the set given by Lemma~\ref{lem:quantelim} 
for our $F$. That is, $|\mathcal{T}|\le \delta^2$ and, 
for every $(c,d)\in \C^2\minus\mathcal{T}$, the set $\gamma_{c,d}$ 
is an algebraic curve of degree at most $\delta$ (or empty).

\paragraph{Dual curves.}
We define, in an analogous manner, a dual system of curves by switching the roles of the coordinates 
$x,y$ and the coordinates $s,t$, as follows. 
For every point $(a,b)\in\C^2$, we define
\[
\gamma^*_{a,b} = \{(s,t)\in \C^2\mid F(a,b,s,t) = 0\}.
\]
As above, Lemma \ref{lem:quantelim} implies 
that there exists an exceptional finite set $\mathcal{S}$ of size at most $\delta^2$, such that for every 
$(a,b)\in \C^2\minus \mathcal{S}$ the set $\gamma_{a,b}^*$ is an algebraic curve of degree at most $\delta$
or the empty set.

Note that $(x,y)\in\gamma_{s,t}$ if and only if $(s,t)\in\gamma_{x,y}^*$, and both are  equivalent to $F(x,y,s,t)=0$.

We will analyze what happens when many of these curves have a large common intersection.
The following definition introduces the terminology for this step, which we will use throughout the analysis.

\begin{definition}\label{def:popcurves}
We say that an irreducible algebraic curve $\gamma\subset\C^2$ is a {\em popular curve} if 
there exist at least $\delta^2+1$ distinct points $(s,t)\in\C^2\minus \mathcal{T}$ such that $\gamma\subset\gamma_{s,t}$. 
We denote by $\mathcal C$ the set of all popular curves.\\[0.1cm]
Similarly, we say that an irreducible algebraic curve $\gamma^*\subset\C^2$ is {\em a popular dual curve} if there exist at least $\delta^2+1$ distinct points $(x,y)\in\C^2\minus \mathcal{S}$ such that $\gamma^*\subset\gamma_{x,y}^*$. We denote by $\mathcal{D}$ the set of all popular dual curves. 
\end{definition}

Informally, the following lemma asserts that the ``finite popularity" in Definition~\ref{def:popcurves}
in fact implies ``infinite popularity".

\begin{lemma}\label{lem:associated}
(a) For every $\gamma\in\mathcal C$ there exists an irreducible algebraic curve $\gamma^*\subset\C^2$ of degree at most $\delta$, 
such that $\gamma\subset \gamma_{s,t}$ for all $(s,t)\in \gamma^*$.\\
(b) For every $\gamma^*\in\mathcal D$ there exists an irreducible algebraic curve $\gamma\subset\C^2$ of degree at most $\delta$, 
such that $\gamma^*\subset \gamma_{x,y}^*$ for all $(x,y)\in \gamma$.
\end{lemma}
\begin{proof}
We prove only part (a) of the lemma, since (b) is fully symmetric.
By definition of $\mathcal{C}$, if $\gamma\in \mathcal{C}$, then there exists a 
set $I\subset \C^2\backslash \mathcal{T}$ of size $|I| = \delta^2+1$ such that 
$\gamma\subset \gamma_{s,t}$ for all $(s,t)\in I$.
This means that, for all $(s,t)\in I$ and for all $(x,y)\in \gamma$, we have
$F(x,y,s,t)=0$, which implies that $(s,t)\in \gamma^*_{x,y}$.
Thus we have $I\subset \gamma^*_{x,y}$ for all $(x,y)\in \gamma$. 

Define
\[S_\gamma := \bigcap_{(x,y)\in \gamma} \gamma^*_{x,y}. \]
Note that for $(x,y)\in \mathcal{S}$ we have $\gamma^*_{x,y} = \C^2$, 
but $\mathcal{S}$ is finite and $\gamma$ is infinite,
so at least one $\gamma^*_{x,y}$ is a curve, and we can safely ignore the pairs $(x,y)\in \mathcal{S}$ in the above intersection.
By the previous paragraph we have $I\subset S_\gamma$.
Since all the curves $\gamma^*_{x,y}$ for $(x,y)\in \gamma\backslash\mathcal{S}$ have degree at most $\delta$, 
Lemma~3.10 of \cite{RSZ15}, mentioned above, implies that $S_\gamma$ either contains a one-dimensional component of degree at most $\delta$, or is finite and consists of at most $\delta^2$ points.
Since $|I|>\delta^2$, the former case must hold.
Let $\gamma^*$ be some irreducible one-dimensional component of $S_\gamma$.

If $(s,t)\in \gamma^*$, then for all $(x,y)\in \gamma$ we have $(s,t)\in \gamma^*_{x,y}$, which by duality implies that $(x,y)\in \gamma_{s,t}$.
Thus for all $(s,t)\in \gamma^*$, we have $\gamma\subset \gamma_{s,t}$, as asserted.
\end{proof}

We refer to $\gamma^*$, which is not necessarily unique, as an {\it associated curve} of $\gamma$.
Note that $\gamma^*$ need not be one of the dual curves $\gamma^*_{x,y}$, but may be only a component of such a curve. 
Nevertheless, we find it convenient to use the star notation also for the associated curves.

\paragraph{Splitting the variables.}
Note that in setting up the curves, we made an arbitrary choice by splitting the four coordinates into the two pairs $x,y$ and $s,t$.
Evidently, since our assumptions on
$F$ are symmetric in the variables $x,y,s,t$, any split of the variables $x,y,s,t$  into two pairs will 
give a set of curves and a set of dual curves with the same properties discussed above.
On the other hand, $F$ itself is not assumed to be symmetric in $x,y,s,t$, and thus certain splits might 
yield better-behaved curves than other splits. 

Note, though, that our analysis handles the first and the second coordinate 
pairs in a fully symmetric manner, 
and that the order of the coordinates in each pair is also irrelevant.
Hence it suffices to consider only the three coordinate splits which,
without loss of generality, are
$\{(x,y), (s,t)\}$, $\{(s,y),(x,t)\}$, and $\{(t,y),(s,x)\}$.
To keep the notation simple, we represent each of these splits 
by a permutation $\sigma$ of $(x,y,s,t)$, 
and we introduce the notation 
\[F_\sigma(x,y,s,t):=F(\sigma(x),\sigma(y),\sigma(s),\sigma(t)),\]
with the understanding that the corresponding split is into the pairs 
$(\sigma(x),\sigma(y))$, and $(\sigma(s),\sigma(t))$.
The three relevant permutations are thus
$(x,y,s,t)$, $(s,y,x,t)$ and $(t,y,s,x)$.
 Clearly, $F$ is of the special form described in Theorem~\ref{thm:main}$(ii)$ if and only if $F_\sigma$ is, for any permutation $\sigma$.
For each $\sigma$, we define the curves (and the dual curves) exactly as above, only with $F_\sigma$ replacing $F$.

The main step in our proof is the following key proposition, which shows that 
{\it for some permutation} $\sigma$ of the coordinates (out of the three that we consider),
we can exclude (or rather control) the popular curves and popular dual curves,
unless $F_\sigma$ (and thus $F$) satisfies property $(ii)$ of Theorem~\ref{thm:main}.
The proof of the proposition is given in Section~\ref{sec:proofofkey}.
Note that its statement is only about $F$ and does not involve the specific sets $A,B,C,D$.

\begin{proposition}\label{prop:keyprop}
Either $F$ satisfies property $(ii)$ of Theorem \ref{thm:main}, or, for some permutation of the coordinates
$x,y,s,t$, both of the following properties hold.\\
$(a)$ There exists a one-dimensional variety $\mathcal{T}'\subset\C^2$ of degree $O(\delta^{4})$ containing $\mathcal{T}$, 
such that for every $(s,t)\in\C^2\minus \mathcal{T}'$, the curve $\gamma_{s,t}$ does not contain a popular curve.\\
$(b)$ There exists a one-dimensional variety $\mathcal{S}'\subset\C^2$ of degree $O(\delta^{4})$ containing $\mathcal{S}$,
such that for every $(x,y)\in\C^2\minus \mathcal{S}'$, the dual curve $\gamma_{x,y}^*$ does not contain a popular dual curve.
\end{proposition}

\subsection{Incidences}
We continue with the analysis, assuming 
that $F$ does not satisfy property $(ii)$ of Theorem \ref{thm:main},
and that Proposition~\ref{prop:keyprop} holds,
for some permutation of $x,y,s,t$.
By relabeling the variables if necessary, we may assume that the 
corresponding coordinate split is the one we have been
using, namely $\{(x,y),(s,t)\}$.

We introduce the following set of points and {\em multiset} of curves (some of the curves may coincide or overlap as point sets):
\[
\Pi:=(A\times B)\minus \mathcal{S'}~~~~~~\text{and}~~~~~~\Gamma := \{\gamma_{c,d}\mid (c,d)\in(C\times D)\minus\mathcal{T'}\},
\]
where $\mathcal T'$ and $\mathcal S'$ are the varieties produced in Proposition~\ref{prop:keyprop}. 
By definition, and as was already pointed out, $(a,b,c,d)\in Z(F)$ if and only if 
$(a,b)\in \gamma_{c,d}$ and $(c,d)\in \gamma^*_{a,b}$.
This lets us relate $M$, the quantity that we want to bound, 
to $I(\Pi,\Gamma)$, the number of incidences between these points and curves;
since curves in $\Gamma$ may coincide or overlap, 
these incidences should be counted with the multiplicity of the relevant curves.

\begin{lemma}\label{lem:ItoM}
The quantity $M:=|Z(F)\cap (A\times B\times C\times D)|$ satisfies
$$
M \le I(\Pi, \Gamma) + 
O\left(|A||B|+|C||D|+|A||C|+|A||D|+|B||C|+|B||D|\right),
$$
where the constant of proportionality depends (polynomially) on the degree $\delta$ of $F$.
\end{lemma}
\begin{proof}
Any $(a,b,c,d)$ in $Z(F)\cap (A\times B\times C\times D)$ that is not counted in $I(\Pi,\Gamma)$ must have
$(a,b)\in\mathcal{S'}$ or $(c,d)\in\mathcal{T'}$. 

We apply the Schwartz-Zippel lemma 
(see \cite[Lemma A.4]{RSZ15})
to the curves $\mathcal S'$ and $\mathcal T'$, each of degree $O(\delta^4)$.
To be precise, we apply the lemma to the purely one-dimensional components of $\mathcal{S}'$, 
and add the number of zero-dimensional components of $\mathcal{S}'$, which,
as follows from the proof of Proposition~\ref{prop:keyprop} (given in Section~\ref{sec:proofofkey}), 
is only $O(\delta^2)$. 
We do the same for $\mathcal{T'}$.
It follows that
$|\mathcal{S'}\cap (A\times B)|=O(\delta^{4}|A|+\delta^{4}|B|)$,
and similarly
$|\mathcal{T'}\cap (C\times D) | = O(\delta^{4}|C|+\delta^{4}|D|)$.

For any $(c,d)\in\mathcal{T}'\minus \mathcal{T}$, we have $|\gamma_{c,d}\cap (A\times B)| =O(\delta|A|+\delta|B|)$, 
and for any $(a,b)\in \mathcal{S}'\minus \mathcal{S}$ we have 
$|\gamma_{a,b}^*\cap (C\times D)| =O(\delta|C|+\delta|D|)$;
both bounds follow from additional applications of
the Schwartz-Zippel lemma.
Thus the contribution from the excluded pairs 
$(a,b)\in \mathcal{S}'\backslash \mathcal{S}$ and $(c,d)\in\mathcal{T}'\backslash\mathcal{T}$
is $O(\delta^{5}(|A|+|B|)(|C|+|D|))$.

For $(c,d)\in \mathcal{T}$ (resp., $(a,b)\in \mathcal{S}$), the set $\gamma_{c,d}$ (resp., $\gamma_{a,b}^*$) could be
all of $\C^2$, so we only have the trivial bounds $|\gamma_{c,d}\cap (A\times B)| =O(|A||B|)$ and $|\gamma_{a,b}^*\cap (C\times D)| =O(|C||D|)$.
Fortunately, we have $|\mathcal{T}| \le \delta^2$ and $|\mathcal{S}|\le \delta^2$,
so the contribution from these remaining pairs is only $O(\delta^2|A||B|+\delta^2|C||D|)$.
\end{proof}

We now define exactly in what sense we require the curves to have well-behaved intersections.

\begin{definition}\label{def:boundedmult}
Let $\Pi$ be a finite set of points in $\C^2$, 
and let $\Gamma$ be a finite multiset of curves in $\C^2$, which can coincide or overlap.
We say that the system $(\Pi,\Gamma)$ has \emph{$(\lambda,\mu)$-bounded multiplicity} 
if\,\footnote{In both (a) and (b) the curves should be counted with their multiplicity.}\\ 
$(a)$ for any curve $\gamma\in \Gamma$,
there are at most $\lambda$ other curves $\gamma'\in \Gamma$ such that 
$|\gamma\cap\gamma'|>\mu$; and\\
$(b)$ for any point $p\in \Pi$, 
there are at most $\lambda$ other points $p'\in \Pi$ such that there are more than $\mu$ curves that contain both $p$ and $p'$.
\end{definition}

We use the following incidence bound, taken from our previous work \cite{RSZ15}, 
where it was deduced from the incidence bound in Solymosi and De Zeeuw~\cite{SZ15}.

\begin{proposition}\label{prop:incbound}
Let $A_1,A_2$ be finite subsets of $\C$ and $\Pi \subset A_1\times A_2$,
and let $\Gamma$ be a finite multiset of algebraic curves in $\C^2$ of degree at most $\delta$, 
such that the system $(\Pi,\Gamma)$ has $(\lambda,\mu)$-bounded multiplicity.
Then 
\[I(\Pi, \Gamma) 
=O\left(|A_1|^{2/3}|A_2|^{2/3}|\Gamma|^{2/3}+|A_1||A_2|+|\Gamma|\right),
\]
where the constant of proportionality depends (polynomially) on $\lambda$ and $\mu$.
\end{proposition}

\begin{lemma}\label{boundedmul}
If $F$ does not satisfy property $(ii)$ of Theorem \ref{thm:main},
then $(\Pi,\Gamma)$ is a system that has $(\delta^3, \delta^2)$-bounded multiplicity.
\end{lemma}
\begin{proof}
By our choice of $\Gamma$, every $\gamma\in \Gamma$ does not contain a popular curve. 
Thus, by definition, each of its at most $\delta$ irreducible components 
is common to at most $\delta^2$ other curves of $\Gamma$.
So we get a total of at most $\delta^3$ curves of $\Gamma$ that share an irreducible component with $\gamma$.
For any other curve $\gamma'\in \Gamma$, which is not one of the at most $\delta^3$ excluded curves,
the intersection $\gamma\cap \gamma'$ contains at most $\delta^2$ points, by B\'ezout's 
theorem. This shows that the property in Definition~\ref{def:boundedmult}(a) holds.

Similarly, by our choice of $\Pi$, for every $p=(x,y)\in \Pi$, the curve $\gamma_{x,y}^*$
does not contain a popular dual curve. Thus, by definition, each of its irreducible components is 
shared by at most $\delta^2$ dual curves $\gamma_{x',y'}^*$, for $p'=(x',y')\in\Pi$.
So we get a total of at most $\delta^3$ points $p'$ of $\Pi$ with this property.
For any other point $p'\in \Pi$, which is not one of the at most $\delta^3$ excluded ones, we have
$|\gamma_{x,y}^*\cap\gamma_{x',y'}^*|\le \delta^2$, by B\'ezout's theorem. But this, by our
definition of dual curves, exactly means that
the number of curves of the form $\gamma_{c,d}$ that pass through both $p$ and $p'$
is at most $\delta^2$. 
This establishes the property in Definition~\ref{def:boundedmult}(b), 
and thus proves the lemma.
\end{proof}

Combining Lemma~\ref{boundedmul} with Proposition~\ref{prop:incbound}, and then with Lemma~\ref{lem:ItoM},
we conclude that
\begin{align*}
|Z(F)\cap (A\times B\times C\times D)| = O\Big(|A|^{2/3}|B|^{2/3}|C|^{2/3}|D|^{2/3}+|A||B|+ |A||C|\\+|A||D|+|B||C|+|B||D|+|C||D|\Big),
\end{align*}
for every $A,B,C,D\subset\C$.
This completes the proof of Theorem~\ref{thm:main} (modulo the still missing proof 
of Proposition~\ref{prop:keyprop}). $\hfill\square$

\section{Proof of Proposition \ref{prop:keyprop}}\label{sec:proofofkey}

\subsection{The varieties \texorpdfstring{$V,W,\widetilde{W}$}{V,W}}

We define the variety
\[
V: = \{(x,y,x',y',s,t)\in \C^6\mid F(x,y,s,t)=0, F(x',y',s,t)=0\}.
\]

\begin{lemma}\label{lem:Vdim4}
The variety $V$ has dimension $4$.
\end{lemma}
\begin{proof}
Recall our assumptions that $F$ is irreducible and that none of $F_x,F_y,F_s,F_t$ is identically zero.
The variety $V$ is not empty, since it contains the point $(x,y,x,y,s,t)$ for any point $(x,y,s,t)\in Z(F)$.
Since $V$ is the common zero set of two nontrivial polynomials in $\C^6$, 
it follows by standard arguments in algebraic geometry (see, e.g., \cite[Lemma A.1]{RSZ15}) that $V$ has dimension either $4$ or $5$.
Moreover, since $F$ is irreducible, it follows that
$V$ is of dimension $5$ if and only if $F(x,y,s,t)\equiv \alpha F(x',y',s,t)$, 
for some constant $\alpha\in \C$, 
where this is interpreted as a polynomial identity in $\C[x,y,x',y',s,t]$.
Noting that, by the assumption of Theorem \ref{thm:main}, the derivative of $ F(x',y',s,t)$ with respect to 
the variable $x$ is identically zero, for any $\alpha$, whereas
the derivative of $F(x,y,s,t)$ with respect to $x$ is not, we conclude that this identity is impossible, and thus $V$ is four-dimensional.
\end{proof}

Let $G$ be the polynomial in $\C[x,y,x',y',s,t]$ given by
$$
G=F_s(x,y,s,t)F_t(x',y',s,t) - F_s(x',y',s,t)F_t(x,y,s,t).
$$

Consider the subvariety $W:=V\cap Z(G)$ of $V$.
Note that in case the set $\mathcal T$ (from Lemma \ref{lem:quantelim}) is non-empty, the variety $W$ (and hence also $V$) contains the subvariety
$V_0:=\C^4\times \mathcal T$, and the latter is four-dimensional, as it is 
the union of a finite number of $4$-flats in $\C^6$.
What is relevant to us in our analysis are the components of $W$ that are not contained in $V_0$.
For this reason we replace $W$ by the subvariety $\widetilde W\subset W$ which is defined as the union of the irreducible components
of $W$ that are not contained in $V_0$.

The following lemma shows the significance of $G$ (and $\widetilde{W}$): 
It serves to detect popular curves.

\begin{lemma}\label{lem:inW}
Let $\gamma$ be a popular curve and let $\gamma^*$ be an associated curve of $\gamma$. 
Then $\gamma\times\gamma\times \gamma^*\subset \widetilde W$.
\end{lemma}
\begin{proof}

Let $\gamma,\gamma^*$ be as in the statement, and consider any pair  of points
$(x,y),(x',y')\in \gamma$. 
By the definition of $\gamma^*$ in Lemma \ref{lem:associated}, we have $\gamma\subset \gamma_{s,t}$ for all $(s,t)\in \gamma^*$,
which by duality gives $\gamma^*\subset \gamma_{x,y}^*\cap\gamma_{x',y'}^*$.
In particular, for each $(s,t)\in \gamma^*$ we have $F(x,y,s,t)=F(x',y',s,t)=0$,
implying that $(x,y,x',y',s,t)\in V$. That is, $\gamma\times \gamma\times\gamma^*\subset V$.
Moreover, proceeding with the same pair $(x,y)$, $(x',y')$ 
if $(s,t)\in\gamma^*$ is a nonsingular point of $\gamma^*_{x,y}$ and $\gamma^*_{x',y'}$, 
then both must have the same tangent line as $\gamma^*$ at $(s,t)$.
Then the vectors 
$(F_s(x,y,s,t),F_t(x,y,s,t))$ and $(F_s(x',y',s,t),F_t(x',y',s,t))$ are parallel,
since they are tangent vectors to $\gamma_{x,y}^*$ and $\gamma_{x',y'}^*$ at $(s,t)$, respectively.
Thus we have
\[G(x,y,x',y',s,t) = \det\begin{pmatrix} 
  F_s(x,y,s,t) & F_t(x,y,s,t)\\
 F_s(x',y',s,t) & F_t(x',y',s,t)
\end{pmatrix}=0.\]
If $(s,t)\in \gamma^*$ is a singular point of 
$\gamma^*_{x,y}$ or of $\gamma^*_{x',y'}$, then the corresponding vector $(F_s(x,y,s,t),F_t(x,y,s,t))$ or
 $(F_s(x',y',s,t),F_t(x',y',s,t))$
 is zero, so the determinant above is also zero.
Therefore, $G(x,y,x',y',s,t)=0$ for all $(x,y),(x',y')\in \gamma$ and $(s,t)\in \gamma^*$,
implying that $\gamma\times \gamma\times \gamma^*\subset W$.
Since $\mathcal T$ is finite, $\gamma^*\cap \mathcal T$ is at 
most finite, and thus $\gamma\times\gamma\times\gamma^*$ cannot be contained in $V_0$.
Hence $\gamma\times \gamma\times \gamma^*\subset \widetilde W$.
\end{proof}

Note that $\widetilde W$ is of dimension at least three, since
$$
\{(x,y,x,y,s,t)\in \C^6\mid F(x,y,s,t)=0\}\subset\widetilde W.
$$
Also, since $\widetilde W\subset W\subset V$, and in view of Lemma~\ref{lem:Vdim4},
its dimension is at most four.

We will show that if $\dim \widetilde{W} =3$, {\it for some} permutation $\sigma$ of the coordinates $x,y,s,t$,
then we can use this to control the popular curves 
and the dual popular curves, as asserted in parts (a) and (b) of Proposition~\ref{prop:keyprop}.
On the other hand, if $\dim \widetilde{W}=4$ {\it for every} choice of $\sigma$, 
then we will deduce that $F$ must have a special form, as in Theorem~\ref{thm:main}$(ii)$

\subsection{The case \texorpdfstring{$\dim \widetilde{W} =  3$}{dim W = 3}}

Assume that $\widetilde{W}$ is of dimension three, for some permutation of the coordinates $x,y,s,t$,
which, without loss of generality, we assume to be the permutation that corresponds to the split $\{(x,y),(s,t)\}$. 
We claim that then properties $(a)$ and $(b)$ of Proposition~\ref{prop:keyprop} hold for that permutation.

Let $\gamma$ be a popular curve and let 
$\gamma^*$ be an associated curve of $\gamma$.
Note that $\gamma\times\gamma\times\gamma^*$ is an irreducible three-dimensional 
algebraic variety, since it is a Cartesian product of three irreducible one-dimensional varieties
(see \cite[Appendix A.1]{RSZ15}).
By Lemma~\ref{lem:inW}, 
$\gamma\times\gamma\times\gamma^*$ is contained in $\widetilde{W}$.
In other words, $\gamma\times\gamma\times\gamma^*$ is one of the three-dimensional 
irreducible components of $\widetilde{W}$.
We note that the variety $W$ containing $\widetilde{W}$ is of degree $O(\delta^3)$, since it is defined by the equations $F(x,y,s,t)=0$, $F(x',y',s,t)=0$, $G=0$, 
of respective degrees $\delta$, $\delta$, $2\delta-2$.
Hence, the number of the three-dimensional irreducible components
of $\widetilde{W}$ is $O(\delta^3)$.


Suppose we have a pair of distinct curves $\gamma_1^*$, $\gamma_2^*$, 
with $\gamma_1^*$ associated to a popular curve $\gamma_1$ and $\gamma_2^*$ associated to a popular curve $\gamma_2$
(but $\gamma_1$ and $\gamma_2$ are not necessarily distinct).
Then the products $\gamma_1\times\gamma_1\times\gamma_1^*$ and $\gamma_2\times\gamma_2\times\gamma_2^*$ are distinct.
It follows that the number of distinct associated curves $\gamma^*$ is bounded by $O(\delta^3)$.

In other words, the points $(s,t)\in \C^2\backslash \mathcal{T}$ for which $\gamma_{s,t}$ contains a popular 
curve are contained in the union of at most $O(\delta^3)$ (associated) curves, each of degree at most $\delta$.
We define $\mathcal{T}''$ to be the union of these $O(\delta^3)$ curves, and set 
$\mathcal{T}':=\mathcal{T}\cup \mathcal{T}''$.
Then $\mathcal{T}'$ is a union of a curve of degree $O(\delta^4)$ and at most $\delta^2$ isolated points.
This proves part $(a)$ of Proposition \ref{prop:keyprop}.

Similarly, for any pair of distinct popular curves $\gamma_1$, $\gamma_2$, with respective associated 
curves $\gamma_1^*,\gamma_2^*$ (distinct or not), the products 
$\gamma_1\times\gamma_1\times\gamma_1^*$ and $\gamma_2\times\gamma_2\times\gamma_2^*$ are distinct.
So, as above, the number of popular curves $\gamma$ is bounded by $O(\delta^3)$.

By Lemma~\ref{lem:associated}(b), 
for every popular dual curve $\gamma^*$ there exists an irreducible algebraic curve $\gamma\subset\C^2$ of degree at most $\delta$, 
such that $\gamma^*\subset \gamma_{x,y}^*$ for all $(x,y)\in \gamma$. 
By duality, this implies that $(x,y)\in \gamma_{s,t}$, for every $(s,t)\in \gamma^*$ and every $(x,y)\in\gamma$.
Thus, $\gamma\subset \gamma_{s,t}$ for every $(s,t)\in \gamma^*$. Hence
$\gamma$ is popular, and $\gamma^*$ is an associated curve of $\gamma$.
To recap, we get that every popular dual curve $\gamma^*$ is an 
associated curve of some popular curve $\gamma$.

Hence, points $(x,y)\in \C^2\backslash \mathcal{S}$ for which $\gamma_{x,y}^*$ contains a popular 
dual curve are contained in the union of $O(\delta^3)$ curves, each of degree at most $\delta$.
We define $\mathcal S''$ to be the union of these $O(\delta^3)$ curves, and set 
$\mathcal{S}':=\mathcal{S}\cup \mathcal{S}''$.
Then $\mathcal{S}'$ is a union of a curve of degree at most $O(\delta^4)$ and of at most $\delta^2$ isolated points.
This proves (b), which completes the proof of Proposition~\ref{prop:keyprop} in the case $\dim \widetilde{W} = 3$.

\subsection{The case \texorpdfstring{$\dim \widetilde{W} = 4$}{dim W = 4}}

Assume that $\widetilde{W}$ is four-dimensional for each of the three permutations $\sigma$ of the coordinates $x,y,s,t$.
Fix a permutation $\sigma$;
by relabeling the variables if necessary we can assume that $\sigma$ is the permutation $\{(x,y), (s,t)\}$.

\begin{lemma}\label{lem:Z0}
There exists a two-dimensional subvariety $Z_0\subset Z(F)$ such that the 
following holds. For every $(x,y,s,t)\in Z(F)\backslash Z_0$, there exist $x',y'\in\C$ such that\\
$(a)$ The point $(x,y,x', y', s,t)$ is a regular point of $\widetilde{W}$.\\
$(b)$ None of the partial derivatives of $F$ vanishes at $(x,y,s,t)$ or at $(x',y',s,t)$.
\end{lemma}
\begin{proof}
Let $W'$ be the set of regular points $(x,y,x',y',s,t)$ of $\widetilde{W}$
that are contained in some irreducible four-dimensional component of $\widetilde{W}$, 
and such that none of the partial derivatives of $F$ vanishes at $(x,y,s,t)$
or at $(x',y',s,t)$.
Since the complement of each of these properties (being singular or having a vanishing derivative)
defines a lower-dimensional subvariety of $\widetilde W$ (as is not hard to verify),
the Zariski closure of $W'$ is four-dimensional. 

Consider the projection $\pi:(x,y,x',y',s,t)\mapsto (x,y,s,t)$.
We have $\pi(W')\subset Z(F)$.
If the Zariski closure of $\pi(W')$ is three-dimensional, then it must be equal to $Z(F)$. 
In this case, there is a two-dimensional subvariety $Z_0\subset Z(F)$ (which might
also be lower-dimensional or empty), such that 
$Z(F)\backslash \pi(W')\subset Z_0$, and we are done.

Suppose then that the Zariski closure of $\pi(W')$ is a subvariety of $Z(F)$ of dimension at most two.
We claim that there exists a point $(x_0,y_0,s_0,t_0)\in \pi(W')$ for which the fiber 
$\pi^{-1}(x_0,y_0,s_0,t_0)\cap \cl(W')$ is at most one-dimensional.
Indeed, by our construction of $\widetilde{W}$ (and since $\cl(W')$ is a four-dimensional irreducible component
of $\widetilde{W}$), the set $W'':=\cl(W')\setminus V_0$ is a 
Zariski-dense open subset of $\cl(W')$,
and $(s,t)\in \C^2\setminus \mathcal T$ for every $(x,y,x',y',s,t)\in \cl(W')$. 
Since $W'$ is also Zariski-dense and open in $\cl (W')$, it means that there exists a point
$(x_0,y_0,x_0',y_0',s_0,t_0)\in W''\cap W'$.
In particular, there exists
$(x_0,y_0,s_0,t_0)\in \pi(W')$ with $(s_0,t_0)\not\in \mathcal T$.
Noting that 
$$
\pi^{-1}(x_0,y_0,s_0,t_0)\cap \widetilde{W}
\subset
 \{(x_0, y_0,x',y',s_0,t_0)\mid F(x',y',s_0,t_0)=0\},
$$
the fiber $\pi^{-1}(x_0,y_0,s_0,t_0)$ is indeed at most one-dimensional. 
But, by \cite[Theorem 11.12]{Ha92}, 
$$
\dim\cl(W')\le  \dim \cl(\pi(W'))+\dim (\pi^{-1}(x_0,y_0,s_0,t_0)\cap \cl(W')),
$$
which implies $\dim\cl(W')\le 3$. This yields a contradiction (to $\pi(W')$ being at most
two-dimensional), and hence completes the proof.
\end{proof}

Recall that in this subsection we assume that $\widetilde{W}$ is four-dimensional 
for each of the three permutations $\sigma$ of the coordinates $x,y,s,t$.
In what follows we will make use of all three permutations, 
so let $\sigma_1$, $\sigma_2$, $\sigma_3$ denote the permutations $\{(x,y),(s,t)\}$, $\{(t,y),(s,x)\}$, $\{(s,y),(x,t)\}$, respectively,
and let us denote the variety $\widetilde{W}$ that corresponds to a permutation $\sigma$ by $\widetilde{W}_\sigma$.

For each $i=1,2,3$, let $Z_{\sigma_i}$ denote the excluded subvariety of dimension at most two, given by applying Lemma~\ref{lem:Z0} to the variety $\widetilde{W}_{\sigma_i}$, and,
by a slight abuse of notation, put $Z_0:=Z_{\sigma_1}\cup Z_{\sigma_2}\cup Z_{\sigma_3}$. 
Evidently, $Z_0\subset Z(F)$ is at most two-dimensional.

Fix $(x_0,y_0,s_0,t_0)\in Z(F)\backslash Z_0$.
Let $U\subset Z(F)\backslash Z_0$ be an open neighborhood of $(x_0,y_0,s_0,t_0)$.
By Lemma~\ref{lem:Z0}, there exist $x_0',y_0'$ such that 
$(x_0,y_0,x_0',y_0',s_0,t_0)$ is a regular point of $\widetilde{W}_{\sigma_1}$
and the partial derivatives of $F$ do not vanish at $(x_0,y_0,s_0,t_0)$ and at $(x_0',y_0',s_0,t_0)$.
In particular, there exists a neighborhood $U_1$ of $(x_0,y_0,x_0',y_0',s_0,t_0)$ in $\widetilde{W}_{\sigma_1}$,
such that every point $(x,y,x',y',s,t)\in U_1$ is a regular point of $\widetilde{W}_{\sigma_1}$,
and satisfies
\begin{align}
F_s(x,y,s,t)F_t(x',y',s,t)&= F_s(x',y',s,t)F_t(x,y,s,t),\nonumber\\
F(x,y,s,t)&=0,\label{Weq}\\
F(x',y',s,t)&=0,\nonumber
\end{align}
and the partial derivatives of $F$ do not vanish at $(x,y,s,t)$ and at $(x',y',s,t)$.
In particular, locally, over $U_1$, the varieties $\widetilde{W}_{\sigma_1}$ and $V$ coincide.

We apply the implicit function theorem to the last two equations in \eqref{Weq} to write
$y=y(x,s,t)$, $y'=y'(x',s,t)$ for  $(x,y,x',y',s,t)\in U_1$. 
(For this we use the fact that (a) $F_y(x,y,s,t)$ and $F_y(x',y',s,t)$ do not
vanish at any $(x,y,x',y',s,t)\in U_1$, and (b) $\widetilde{W}_{\sigma_1}\equiv V$ over $U_1$.)
Then (with a suitable reshuffling of the coordinates) $U_1$ is the graph of the function $(x,x',s,t)\mapsto (y(x,s,t),y'(x',s,t))$, 
over the open (i.e., four-dimensional) domain $\tau(U_1)\subset \C^4$,
where $\tau$ is the projection $(x,y,x',y',s,t)\mapsto (x,x',s,t)$
(the non-vanishing of the partial derivatives is easily seen to imply that
$\tau(U_1)$ is indeed four-dimensional).

The first equation in \eqref{Weq} gives (note that the denominators do not vanish)
$$
\frac{F_s(x,y(x,s,t),s,t)}
{F_t(x,y(x,s,t),s,t)}=\frac{F_s(x',y'(x',s,t),s,t)}{F_t(x',y'(x',s,t),s,t)},
$$
for each $(x,x',s,t)\in \tau(U_1)$.
It follows that both sides of the equation are independent of $x,x'$
(varying $x$ on the left-hand
side does not change the right-hand side, and vice versa).
Thus we can write 
$$
\frac{F_s(x,y(x,s,t),s,t)}
{F_t(x,y(x,s,t),s,t)},
$$
as a function $h_1(s,t)$ (independent of $x$),
which can also be rewritten as
$$
h_1(s,t)=\left(
-\frac{F_s(x,y(x,s,t),s,t)}{F_y(x,y(x,s,t),s,t)}\right) \Big{/} \left(
-\frac{F_t(x,y(x,s,t),s,t)}{F_y(x,y(x,s,t),s,t)}\right),
$$
or, using implicit differentiation and writing partial derivatives of $y=y(x,s,t)$ in the form $y_x,y_s,y_t$,
\begin{equation}\label{eq:h1}
h_1(s,t) = \frac{y_s(x,s,t)}{y_t(x,s,t)}.
\end{equation}

As above, let $\pi:\C^6\to \C^4$ denote the projection $(x,y,x',y',s,t)\mapsto (x,y,s,t)$.
By shrinking $U$ or $U_1$, if needed, we may assume that $\pi(U_1)=U$.
So $U$ is an open neighborhood of $(x_0,y_0,s_0,t_0)$ in $Z(F)$
and forms the graph of the function $y=y(x,s,t)$ (again, with a suitable reshuffling of the coordinates) over the open domain $\rho(U)\subset \C^3$,
where $\rho:\C^4\to \C^3$ denotes the projection $(x,y,s,t)\mapsto (x,s,t)$.
Then the function $h_1$, as represented in \eqref{eq:h1}, is defined over the domain $\rho(U)$ and is independent of the variable $x$.

We can do the same for the permutations $(t,y, s,x)$ and $(s,y, t,x)$ (although we permute the roles 
of the variables, we will keep listing them in the same order as the variables of $F$, to limit the confusion).
In appropriate neighborhoods $U_2\subset \widetilde{W}_{\sigma_2}$ and $U_3\subset\widetilde{W}_{\sigma_3}$
we get
\begin{align*}
F_s(x,y,s,t)F_x(x,y',s,t')&=F_s(x,y',s,t')F_x(x,y,s,t),\\
F_x(x,y,s,t)F_t(x,y',s',t)&=F_x(x,y',s',t)F_t(x,y,s,t),
\end{align*}
so, by shrinking $U_2$ and $U_3$, if needed, we may assume that $\pi(U_2)=\pi(U_3)=U$, and conclude that we can write
\begin{equation}\label{eq:h2}
h_2(s,x) = \frac{y_s(x,s,t)}{y_x(x,s,t)}
\end{equation}
for $(x,s,t)\in\rho(U)$, 
and similarly, 
we can write
\begin{equation}\label{eq:h3}
h_3(t,x) = \frac{y_t(x,s,t)}{y_x(x,s,t)}
\end{equation}
for $(x,s,t)\in\rho(U)$.

In what follows, let us restrict ourselves to the open set $\rho(U)\subset\C^3$.
From \eqref{eq:h1}, \eqref{eq:h2}, and \eqref{eq:h3} we have 
\begin{equation}\label{eq:h2h1h3}
h_2(s,x)=h_1(s,t)h_3(t,x),
\end{equation}
 so we see that
$$
\frac{y_s(x,s,t)}{y_x(x,s,t)} = h_2(s,x) = h_1(s,t)h_3(t,x)
$$
is independent of $t$.
Thus, we can substitute any value of $t$ that occurs in $\rho(U)$, say $t=t_0$, and get
\begin{equation}\label{eq:qp}
\frac{y_s(x,s,t)}{y_x(x,s,t)}
=\frac{h_1(s,t_0)}{1/h_3(t_0,x)}
= \frac{q'(s)}{p'(x)},
\end{equation}
where $p(x):=\int \frac{1}{h_3(t_0,x)}dx$ and $q(s):=\int h_1(s,t_0)ds$ 
(the arbitrary constants in these indefinite integrals clearly do not matter).

In a similar manner,  we see that
$$
\frac{y_s(x,s,t)}{y_t(x,s,t)} = h_1(s,t)=\frac{h_2(s,x)}{h_3(t,x)}
$$
is independent of $x$, so substituting $x=x_0$, say, we get
$$
\frac{y_s(x,s,t)}{y_t(x,s,t)}=\frac{\widehat{q}'(s)}{r'(t)},
$$
where $\widehat{q}(s):=\int h_2(s,x_0)ds$ and $r(t):=\int h_3(t,x_0)dt$.
However, by \eqref{eq:h2h1h3}, we have $h_2(s,x_0)=h_1(s,t_0)h_3(t_0,x_0)$,
so
$$
\widehat{q}(s):=\int h_2(s,x_0)ds= h_3(t_0,x_0)\int h_1(s,t_0)ds= h_3(t_0,x_0)q(s)
$$
(up to an additive constant, which we may assume to be zero).
Therefore, we can redefine $r(t):=h_3(t_0,x_0)\int h_3(t,x_0)dt$, and get
\begin{equation}\label{eq:qr}
\frac{y_s(x,s,t)}{y_t(x,s,t)}=\frac{q'(s)}{r'(t)}.
\end{equation}

Combining \eqref{eq:qp} and \eqref{eq:qr}, we get
\begin{equation}\label{eq:diffeq}
\frac{y_x(x,s,t)}{p'(x)}\equiv\frac{y_s(x,s,t)}{q'(s)}\equiv \frac{y_t(x,s,t)}{r'(t)},
\end{equation}
for all $x,s,t$ such that $(x,y(x,s,t),s,t)$ is in the neighborhood $U$.

We change variables to $u=p(x), v=q(s), w=r(t)$, so that the equations in \eqref{eq:diffeq} become
$$
y_u(u,v,w)\equiv y_v(u,v,w)\equiv y_w(u,v,w).
$$
We perform a second change of variables 
\[u'=u+v+w, ~~~v'=v,~~~w'=w;\]
the inverse change of variables is $u=u'-v'-w', v=v', w=w'$.
Then we have 
$$
y_{v'}=y_u\frac{\partial u}{\partial v'}+y_v\frac{\partial v}{\partial v'}+y_w\frac{\partial w}{\partial v'}
=-y_u+y_v\equiv 0,
$$
and 
$$
y_{w'}=y_u\frac{\partial u}{\partial w'}+y_v\frac{\partial v}{\partial w'}+y_w\frac{\partial w}{\partial w'}
=-y_u+y_w\equiv 0.
$$
In other words, $y$ is independent of $v',w'$, so it depends only on $u'=u+v+w=p(x)+q(s)+r(t)$,
and thus has the form 
$$
y(x,s,t)=h(p(x)+q(s)+r(t)).
$$
By shrinking the neighborhood $U$ further if necessary, we can assume that $h$ is invertible
over the set $\{p(x)+q(s)+r(t)\mid (x,s,t)\in\rho(U)\}$, 
and we can do it since the derivative of $h$ at $p(x_0)+q(s_0)+r(t_0)$ is nonzero.
Hence, for all $(x,y,s,t)\in U$ we have
$$
p(x)-h^{-1}(y)+q(s)+r(t)=0.
$$
This is the form in property $(ii)$ of Theorem \ref{thm:main},
so Proposition \ref{prop:keyprop} is proved also in the case $\dim \widetilde{W} = 4$.
\qed

\section{Curves with many coplanar quadruples}\label{sec:app}

We now give the proof of Theorem \ref{thm:coplanarintro}, 
which we restate below. 
The proof is based on the following structural result, 
proved by the authors \cite[Theorem 6.1]{RSZ15}.
Recall that a tuple of points is \emph{proper} if no two of the points coincide.

\begin{theorem}[Raz, Sharir, and De Zeeuw~\cite{RSZ15}]\label{thm:threecurves}
Let $C_1,C_2,C_3$ be three (not necessarily distinct) irreducible algebraic curves of degree at most $d$ in $\C^2$,
and let $S_1\subset C_1, S_2\subset C_2,S_3\subset C_3$ be finite subsets of size $n$.
Then the number of proper collinear triples in $S_1\times S_2\times S_3$ is $O\left(n^{11/6}\right)$,
where the constant of proportionality depends on $d$,
unless $C_1\cup C_2 \cup C_3$ is a line or a cubic curve.
\end{theorem}

\noindent{\bf Theorem~\ref{thm:coplanarintro}.}
{\it Let $C$ be an algebraic curve of degree $d$ in $\C^3$, and let $S\subset C$ be a finite set of size $n$. 
Then the number of proper coplanar quadruples from $S$ 
is
$O(n^{8/3})$, 
where the constant of proportionality depends on $d$,
unless $C$ contains either a curve that is contained in a plane, or a curve of degree four.
}
\begin{proof}
We define a polynomial $H$ of the 12 variables $x_i,y_i,z_i$, $i=1,2,3,4$, by
\[H(x_1,\ldots,z_4)
= \left|\begin{matrix} 1&x_1&y_1&z_1\\1&x_2&y_2&z_2\\1&x_3&y_3&z_3\\1&x_4&y_4&z_4
\end{matrix}\right|, \]
and a variety $X\subset \C^{12}$, with coordinates $(x_1,y_1,z_1,x_2,y_2,z_2,x_3,y_3,z_3,x_4,y_4,z_4)$, by
\[X:= (C\times C\times C\times C)\cap Z(H), \]
which is the set of all (proper or improper) coplanar quadruples in $C\times C\times C\times C$.
Note that every irreducible component of $C\times C\times C\times C$ is a product of four irreducible components of $C$, and every such product is an irreducible four-dimensional variety.
Therefore, every component of $X$ is three-dimensional, 
unless $H$ vanishes on a component of $C\times C\times C\times C$ (see, e.g., \cite[Lemma A.1]{RSZ15}),
which is a product $C_1\times C_2\times C_3\times C_4$ of four 
(not necessarily distinct) irreducible components $C_1,C_2,C_3,C_4$ of $C$.
If $H$ vanishes on $C_1\times C_2\times C_3\times C_4$, 
then every quadruple from $C_1\times C_2\times C_3\times C_4$ is coplanar, 
which implies that $C$ contains a planar curve, in which case we are done.
We may therefore assume that every component of $X$ is three-dimensional.

By applying a generic rotation in $\C^3$ at the start of the proof, we can assume that $(a)$ no 
two points of $S$ have the same $x$-coordinate, and $(b)$ every $x\in\C$ is the $x$-coordinate 
of at most finitely many points on $C$.
Then the projection $\pi:\C^{12}\to\C^4$ defined by 
\[\pi(x_1,y_1,z_1,x_2,y_2,z_2,x_3,y_3,z_3,x_4,y_4,z_4)= (x_1,x_2,x_3,x_4)\]
is injective on the Cartesian product $S\times S\times S\times S$, and
the image of $S\times S\times S\times S$ is a Cartesian product $A\times A\times A\times A$, 
where $A\subset \C$  is the set of the $x$-coordinates of the points of $S$, and is of size $n$.

Every irreducible component of the Zariski closure $\cl(\pi(X))$ of $\pi(X)$ is a three-dimensional variety.
Indeed, $\cl(\pi(X))$ is at most three-dimensional since $X$ is three-dimensional,
and every component of $\cl(\pi(X))$ is at least three-dimensional,
since $\pi$ has finite preimages in $X$ and every component of $X$ is three-dimensional.
More precisely, 
due to property $(b)$ of the generic rotation,
every point $(x_1,x_2,x_3,x_4)\in \C^4$ has finite preimage $\pi^{-1}(x_1,x_2,x_3,x_4)\cap X$, which by \cite[Theorem 11.12]{Ha92}
implies that every irreducible component of $\cl(\pi(X))$ has the same dimension as the corresponding component of $X$.

The variety $X$ contains all coplanar quadruples in $C\times C\times C\times C$, including all improper quadruples. These are mapped onto the union of the six hyperplanes $x_1=x_2$, $x_1=x_3$, 
$x_1=x_4$, $x_2=x_3$, $x_2=x_4$, and $x_3=x_4$ in $\C^4$.
We remove these hyperplanes from $\cl(\pi(X))$, and we denote the Zariski closure of the remainder by $Y$.
If we write $M$ for the number of proper coplanar quadruples in $S\times S\times S\times S$,
then these $M$ quadruples are mapped to $M$ points in the intersection of $Y$ with the Cartesian product $A\times A\times A\times A$.

Thus we can apply Theorem \ref{thm:main} to each irreducible 
component of $Y$ to bound $M$,
noting that each such component, being an irreducible three-dimensional variety in 
$\C^4$, is the zero set of some irreducible polynomial in four variables
(whose degree depends on $d$). This gives the bound 
in the statement of Theorem \ref{thm:coplanarintro}, 
unless condition $(ii)$ of Theorem \ref{thm:main} holds 
on some irreducible component of $Y$.
Suppose $Y'$ is such a component, so condition $(ii)$ 
gives, for $i=1,2,3,4$, a point $t_i\in\C$,
a neighborhood $D_i$ of $t_i$, and a one-to-one 
analytic map $\phi_i:D_i\to \C$, such that $(t_1,t_2,t_3,t_4)\in Y'$, and 
for each $(x,y,z,w)\in D_1\times D_2\times D_3\times D_4$, 
$(x,y,z,w)\in Y'$ 
if and only if $\phi_1(x)+\phi_2(y)+\phi_3(z)+\phi_4(w)=0$.
By shrinking the $D_i$'s as needed, we can assume that 
$D_1\times D_2\times D_3\times D_4$ does not contain any 
points that were added to 
$\pi(X)$ when taking the closure,
nor does it meet any of the six excluded hyperplanes $x_i=x_j$.

Write $\pi_i$ for the projection $(x_i,y_i,z_i)\mapsto x_i$, for $i=1,2,3,4$.
We choose $p_i\in C$ so that $\pi_i(p_i) = t_i$ (by construction, such points exist).
We also pick an open neighborhood 
$U_i$ of $p_i$ in $C$ so that $\pi_i(U_i)\subset D_i$, 
and we define the analytic map $\varphi_i := \phi_i\circ \pi_i:U_i\to \C$.
By shifting the quadruple $(p_1,p_2,p_3,p_4)$ slightly 
within $U_1\times U_2\times U_3\times U_4$, 
so that each point $p_i$ is shifted independently,
we can assume 
that it is proper, since improper quadruples lie in a lower-dimensional subset of $Y'$.
By shrinking each $U_i$, if needed, we can assume that the corresponding map 
$\varphi_i$ is one-to-one.
We can also assume that each neighborhood 
$U_i$ lies within one irreducible component $C_i\subset C$, and that 
$U_i\cap U_j=\emptyset$ for $i\neq j$ 
(note that the curves $C_i$ need not be distinct, but this property can still be made
to work since the quadruple is proper).
Finally, we can assume that each $U_i$ contains no singular point of $C$, 
since the
set of singular points of $C$ is discrete.
Altogether, 
we have four points
$p_i\in U_i\subset C_i\subset C$, for $i=1,2,3,4$,
satisfying $\varphi_1(p_1)+\varphi_2(p_2)+\varphi_3(p_3)+\varphi_4(p_4) = 0$, 
and a quadruple $(q_1,q_2,q_3,q_4)\in U_1\times U_2\times U_3\times U_4$ 
is coplanar if and only if $\varphi_1(q_1)+\varphi_2(q_2)+\varphi_3(q_3)+\varphi_4(q_4) = 0$.

By continuity, we can choose three sequences $(q_1^{j})_{j=0}^n$, 
$(q_2^{k})_{k=0}^n$, $(q_3^{\ell})_{\ell=0}^{2n}$, such that $q_1^0=p_1$,
$q_2^0=p_2$, $q_3^0=p_3$, $q_1^j\in U_1$ for every $j$, $q_2^k\in U_2$ for every $k$,
$q_3^\ell\in U_3$ for every $\ell$, and  the three sequences $(\varphi_1(q_1^j))_j$,
$(\varphi_2(q_2^k))_k$, $(\varphi_3(q_3^\ell))_\ell$ are arithmetic progressions, where the 
first two have the same common difference $\delta$, and the third has difference $-\delta$.
We therefore have, for each $j,k=0,\ldots,n$, 
\begin{align*}
\varphi_1(q_1^j)+\varphi_2(q_2^k)+\varphi_3(q_3^{j+k})+\varphi_4(p_4) &=\\
(\varphi_1(p_1)+j\delta)+(\varphi_2(p_2)+k\delta)
&+(\varphi_3(p_3)-(j+k)\delta)+\varphi_4(p_4)=0.
\end{align*}
%
We conclude that the quadruple $(q_1^j,q_2^k,q_3^\ell,p_4)$ is coplanar 
for every triple of indices $(j,k,\ell)$ satisfying $\ell=j+k$. That is,
we have found $\Theta(n^2)$ triples $(q_1^j,q_2^k,q_3^{j+k})$
that span a plane in $\C^3$ that goes through the point $p_4$.

Let $\rho$ denote the central projection from the point $p_4$ onto some generic plane $P$ in $\C^3$.
Then $\rho$ maps every triple 
$(q_1^j,q_2^k,q_3^{j+k})$ to a proper 
collinear triple from $\rho(C_1)\times \rho(C_2)\times \rho(C_3)\subset P^3$,
resulting in $\Theta(n^2)$ proper collinear triples.
Applying Theorem~\ref{thm:threecurves}, 
we get that $\rho(C_1)\cup \rho(C_2)\cup \rho(C_3)$ is a line or a cubic curve.

Assume first that $\rho(C_1)\cup \rho(C_2)\cup \rho(C_3)$ is a line.
Then, for each $i=1,2,3$,
$\rho(C_i)$ is a line, 
so $C_i$ is contained in the preimage under $\rho$ of a line.
In other words, each $C_i$ is a planar curve, so in this case we are done.


Next, assume that $\rho(C_1)\cup \rho(C_2)\cup \rho(C_3)$ is a cubic curve.
If two of the curves $\rho(C_1), \rho(C_2), \rho(C_3)$ are distinct, then one of the
three curves must be a line (since the total degree is three).
Then $C$ would contain a planar curve, which again completes the proof. 
Thus we may assume that $\rho(C_1)= \rho(C_2)=\rho(C_3)$.
Moreover, we can repeat this for any $p_4$ within a small neighborhood, 
which implies $C_1=C_2=C_3$.
By repeating the argument with $p_3$ in the role of $p_4$, we can conclude that $C_1=C_2=C_3=C_4$.
Let $C'$ be this curve and note that $\rho(C')$ has degree three.

By \cite[Example 18.16]{Ha92},
projecting a curve from one of its regular points to a generic plane gives a curve whose degree is smaller by one. 
Since $p_4$ was assumed to be a regular point of $C'$,
the fact that $\rho(C')$ has degree three implies that $C'$ has degree four.
\end{proof}

\paragraph{Constructions.}
Planar curves are trivial exceptions to the bound in Theorem \ref{thm:coplanarintro},
but quartic curves are more interesting.
By \cite[Exercise IV.3.6]{Ha77},
a quartic in $\C^3$ is either planar, an intersection of two quadric surfaces, or a rational quartic (a non-planar curve of degree four that can be parameterized by rational functions).
We show that in the second case, there are configurations of $n$ points on the curve that span $\Theta(n^3)$ proper coplanar quadruples.
For the third case, when the curve is a rational quartic, we do not know whether such configurations are possible, and we leave this as an open question.

Let $C$ be a nonsingular and irreducible curve in projective space that is an intersection of quadric surfaces.
The key fact is that such a curve has genus one, i.e., it is an elliptic curve.
This implies that the curve has a group structure,
and it turns out that this group structure is related to coplanarity\footnote{The 
third author would like to thank Mehdi Makhul and Josef Schicho for making him aware of this fact.}.
See for instance \cite[Exercise 3.10]{S09} or \cite[Section 0.8]{H04}.
Specifically, 
if $C$ is a nonsingular irreducible intersection of quadrics in projective space,
there is a group operation $\oplus$ on $C$, and an identity element $\mathcal{O}\in C$ (which may or may not lie at infinity), 
such that four points $p,q,r,s\in C$ are coplanar if and only if $p\oplus q\oplus r\oplus s = \mathcal{O}$.
This even holds when counting with multiplicity. 
For example, if a plane intersects $C$ in $p$ with multiplicity two and in $q,r$ with multiplicity one, 
then $p\oplus p\oplus q\oplus r = \mathcal{O}$, and vice versa.

Because $C$ is an elliptic curve, the group on $C$ is isomorphic to a complex torus $\C/\Lambda$ for a lattice $\Lambda$ (\cite[Corollary VI.5.1.1]{S09}), so in particular it has finite subgroups of any size.
For any $n$, let $H$ be a subgroup of $C$ of size $n$ in the group structure above. 
For any three distinct points $p,q,r\in H$, 
the plane spanned by $p,q,r$ intersects $C$ at the unique point $s$ that satisfies $p\oplus q\oplus r\oplus s = \mathcal{O}$,
and, since $H$ is a subgroup, this equation implies $s\in H$.
Thus $p,q,r,s$ is a coplanar quadruple, and it is proper unless $s$ equals one of $p,q,r$.
The number of improper quadruples of the form $(p',p',q',r')$ that are coplanar (on a plane that intersects $C$ with multiplicity two at $p'$)
equals the number of solutions of $p'\oplus p'\oplus q'\oplus r' = \mathcal{O}$,
which is $O(n^2)$.
Therefore, the point set $H$ spans $\Theta(n^3)$ proper coplanar quadruples.

Given a nonsingular irreducible intersection of quadrics in \emph{affine} space, 
we can take its projective closure and apply the above to get a finite subset $H$ with $\Theta(n^3)$ proper coplanar quadruples.
If any points of $H$ lie at infinity,
then we remove them,
which gives a construction in affine space with asymptotically the same number of coplanar quadruples.

Similar constructions are possible on singular irreducible curves that are intersections of quadrics,
but we omit the details.
We refer to Muntingh \cite[Section 4.4.6]{M10} for a classification of all intersections of quadrics, 
and a careful analysis of group-like structures on them  (the purpose of the constructions there is different from ours, but the sets obtained there also span $\Theta(n^3)$ coplanar quadruples).
On reducible curves, we can easily construct sets of $n$ points with $\Omega(n^4)$ coplanar quadruples.
Indeed, a reducible quartic must contain a line or a conic, 
which lies on a plane, 
so we can choose any $n$ points on that line or conic.


\paragraph{Four-point circles.} Finally, we give one corollary of Theorem \ref{thm:coplanarintro}.
It is related to the orchard problem for circles, which asks for the maximum number of four-point 
circles (circles incident to at least four points) determined by $n$ points in $\R^2$.
Lin et al. \cite{LMNSSZ16} proved that this maximum is $\frac{1}{24}n^3 -O(n^2)$, and 
that any point set attaining that bound must be contained in a quartic curve.
We obtain the following characterization, under the restriction that the point set under consideration 
lies on some constant-degree algebraic curve.
See \cite[Section 4]{LMNSSZ16} for constructions with $\Theta(n^3)$ 
four-point circles on certain curves of degree two, three, and four. 

\begin{corollary}\label{cor:circles}
Let $C$ be an algebraic curve in $\C^2$ of degree $d$,
and let $S\subset C$ be a finite set of size $n$.
Then the number of four-point circles spanned by $S$ is $O(n^{8/3})$, 
where the constant of proportionality depends on $d$,
unless $C$ contains a curve of degree at most four.
\end{corollary} 
\begin{proof}
Let $\varphi :\C^2\to \C^3$ be the map defined by $\varphi(x,y) = (x,y,x^2+y^2)$.
Under $\varphi$, 
a circle in $\C^2$ corresponds to an intersection of the paraboloid $\varphi(\C^2)=Z(z - x^2-y^2)$ 
with a non-vertical plane (intersections with vertical planes correspond to lines in $\C^2$).
Thus four points $p,q,r,s$ in $\C^2$ lie on a common circle if and only if 
$\varphi(p),\varphi(q),\varphi(r),\varphi(s)$ lie on a non-vertical plane.

The image $\varphi(C)$ is an algebraic curve in $\C^3$ of degree at most $2d$.
By Theorem \ref{thm:coplanarintro},
$\varphi(S)$ spans $O(n^{8/3})$ coplanar quadruples, and thus $S$ spans $O(n^{8/3})$ four-point circles,
unless $C$ contains a curve $C'$ such that $\varphi(C')$ is planar or quartic.
If $\varphi(C')$ is planar, then $C'$ is a circle or a line.
If $\varphi(C')$ is quartic, 
then, since $C'$ is the image of $\varphi(C')$ under the projection $(x,y,z)\mapsto (x,y)$, $C'$ has degree at most four.
\end{proof}


\begin{thebibliography}{10}

\bibitem{B16}
S. Ball,
On sets defining few ordinary planes,
in {\tt arXiv:1606.02138} (2016).

\bibitem{Char14}
M. Charalambides,
Distinct distances on curves via rigidity,
\emph{Discrete Comput. Geom.} {\bf 51} (2014), 666--701.

\bibitem{ER00}
G. Elekes and L. R\'onyai,
A combinatorial problem on polynomials and rational functions,
\emph{J. Combinat. Theory Ser. A} {\bf 89} (2000), 1--20.

\bibitem{ES12}
G. Elekes and E. Szab\'o, 
How to find groups? (And how to use them in Erd\H os geometry?), 
\emph{Combinatorica} {\bf 32} (2012), 537--571.


\bibitem{Ha92}
J. Harris,
\emph{Algebraic Geometry: A First Course},
Springer-Verlag, New York, 1992.

\bibitem{Ha77}
R. Hartshorne,
\emph{Algebraic Geometry},
Springer-Verlag, New York, 1977.

\bibitem{H04}
D. Husem\"oller,
\emph{Elliptic Curves},
2nd edition, Springer-Verlag, New York, 2004.

\bibitem{K09}
G. K\'arolyi,
Incidence geometry in combinatorial arithmetic,
in memoriam Gy\"orgy Elekes,
\emph{Ann. Univ. Sci. Budap. Rolando E\"otv\"os, Sect. Math.} {\bf 52} (2009), 37--43.

\bibitem{LMNSSZ16}
A. Lin, M. Makhul, H. Nassajian Mojarrad, J. Schicho, K. Swanepoel, and F. de Zeeuw,
On sets defining few ordinary circles,
in {\tt arXiv:1607.06597} (2016).

\bibitem{M10}
G. Muntingh,
\emph{Topics in Polynomial Interpolation Theory},
Ph.D. dissertation, University of Oslo, 2010.

\bibitem{MRNS15}
B. Murphy, O. Roche-Newton, and I. Shkredov,
Variations on the sum-product problem,
\emph{SIAM J. Discrete Math.} {\bf 29} (2015), 514--540.

\bibitem{NPVZ15}
H. Nassajian Mojarrad, T. Pham, C. Valculescu, and F. de Zeeuw,
Schwartz-Zippel bounds for two-dimensional products,
in {\tt arXiv:1507.08181} (2015).

\bibitem{PS98}
J. Pach and M. Sharir,
On the number of incidences between points and curves,
\emph{Combinat. Probab. Comput.} {\bf 7} (1998), 121--127.

\bibitem{RSS14}
O. E. Raz, M. Sharir, and J. Solymosi,
Polynomials vanishing on grids: The Elekes-R\'onyai problem revisited,
{\it Amer. J. Math.} {\bf 138} (2016), 1029--1065.

\bibitem{RSZ15}
O. E. Raz, M. Sharir, and F. de Zeeuw, 
Polynomials vanishing on Cartesian products: The Elekes-Szab\'o Theorem revisited,
{\it Duke Math. J.}, in press.
Also in {\tt arXiv:1504.05012} (2015).

\bibitem{RN16}
O. Roche-Newton,
A new expander and improved bounds for $A(A+A)$,
in {\tt arXiv:1603.06827} (2016).

\bibitem{Sch80}
J.~Schwartz,
Fast probabilistic algorithms for verification of polynomial identities,
{\it J. ACM} {\bf 27} (1980), 701--717.

\bibitem{SSZ13}
R. Schwartz, J. Solymosi, and F. de Zeeuw,
Extensions of a result of Elekes and R\'onyai,
{\it J. Combinat. Theory Ser. A} {\bf 120} (2013), 1695--1713.

\bibitem{S09}
J. Silverman,
\emph{The Arithmetic of Elliptic Curves},
2nd edition, Springer-Verlag, New York, 2009.

\bibitem{SZ15}
J. Solymosi and F. de Zeeuw,
Incidence bounds for complex algebraic curves on Cartesian products,
in {\tt arxiv:1502.05304} (2015).

\bibitem{TV06}
T. Tao and V. Vu,
{\it Additive Combinatorics},
Cambridge University Press, 
Cambridge, 
2006.

\bibitem{Z16}
F. de Zeeuw,
A survey of Elekes-R\'onyai-type problems,
in {\tt arXiv:1601.06404} (2016).

\bibitem{Zi89}
R. Zippel,
An explicit separation of relativised random polynomial time and relativised deterministic polynomial time,
{\it Inform. Process. Lett.} {\bf 33} (1989), 207--212.

\end{thebibliography}
\end{document}